\documentclass{amsart}
\usepackage{amsmath,amsthm,amssymb}
\usepackage{hyperref}
\usepackage{pstricks,pst-node,multido,pst-coil}

\usepackage{a4wide}
 
\title{Cyclic sieving phenomenon on annular noncrossing permutations}
\author{Jang Soo Kim}

\date{\today}

\newtheorem{thm}{Theorem}[section]
\newtheorem{lem}[thm]{Lemma}
\newtheorem{prop}[thm]{Proposition}
\newtheorem{cor}[thm]{Corollary}
\theoremstyle{definition}

\newtheorem{defn}{Definition}

\theoremstyle{remark}
\newtheorem{remark}{Remark}

\newcommand\flr[1]{\left\lfloor #1\right\rfloor}
\newcommand\ceil[1]{\left\lceil #1\right\rceil}

\newcommand\Qbinom[3]{\genfrac{[}{]}{0pt}{}{#1}{#2}_{#3}}
\newcommand\qbinom[2]{\Qbinom{#1}{#2}{q}}
\newcommand\A{\mathcal{A}}
\newcommand\B{\mathcal{B}}

\newcommand\anc{\operatorname{ANC}}
\newcommand\Par{\operatorname{Par}}
\newcommand\Kre{\operatorname{Kre}^{\mathrm{ann}}}
\newcommand\Nara{\operatorname{Nara}^{\mathrm{ann}}}
\newcommand\Cat{\operatorname{Cat}^{\mathrm{ann}}}

\newcommand\N{\mathbb{N}}
\newcommand\Z{\mathbb{Z}}

\psset{gridlabels=0pt,subgriddiv=1, gridwidth=.1pt,gridcolor=gray}
\psset{linewidth=1pt, unit=.6cm}
\psset{unit=12pt}

\def\evput#1#2#3#4{\rput(#1;#2){\cnode*(0,0){2pt}{#3}} 
\uput{#1.5}[#2](0,0){#4}}
\def\ivput#1#2#3#4{\rput(#1;#2){\cnode*(0,0){2pt}{#3}} 
\uput{#1}[#2](0,0){\rput(-.8;#2){#4}}}
\def\edge#1#2{\ncline{->}{#1}{#2}}

\begin{document}

\begin{abstract}
  We show an instance of the cyclic sieving phenomenon on annular noncrossing
  permutations with given cycle types.  We define annular $q$-Kreweras numbers,
  annular $q$-Narayana numbers, and annular $q$-Catalan number, all of which are
  polynomials in $q$. We then show that these polynomials exhibit the cyclic
  sieving phenomenon on annular noncrossing permutations. We also show that a
  sum of annular $q$-Kreweras numbers becomes an annular $q$-Narayana number and
  a sum of $q$-Narayana numbers becomes an annular $q$-Catalan number.
\end{abstract}

\maketitle

\section{Introduction}

Let $\pi$ be a permutation of $[n]=\{1,2,\dots,n\}$. One can represent $\pi$
inside a disk as shown in Figure~\ref{fig:disk}.  If the arrows of the diagram
of $\pi$ are noncrossing and if every cycle of $\pi$ is oriented clockwise, then
$\pi$ is called a \emph{noncrossing permutation}. If we replace each cycle with
a block, then we get a bijection from noncrossing permutations and noncrossing
partitions.

\begin{figure}
  \centering
\begin{pspicture}(-6,-6)(6,6) 
\psset{linewidth=1.5pt, nodesep =2pt}
\pscircle[linecolor=gray](0,0){5}
\multido{\n=90+-36,\i=1+1}{10}{\evput{5}{\n}{\i}{\i}}
\edge13 \edge34 \edge45 \edge51
\edge97 \edge78 \edge89
\nccircle[angleA=150]{<-}{2}{4pt}
\ncarc[arcangle=20]{->}{10}{6}
\ncarc[arcangle=20]{->}{6}{10}
\end{pspicture}  
\caption{A representation of the permutation $(1,3,4,5)(2)(6,10)(7,8,9)$ inside
  a disk.}
  \label{fig:disk}
\end{figure}
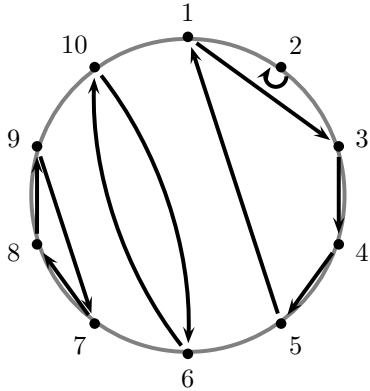

It is well known that the number of noncrossing permutations of $[n]$ is the
\emph{Catalan number}
\[
\mathrm{Cat}(n) = \frac{1}{n+1} \binom{2n}{n},
\]
and the number of noncrossing permutations of $[n]$ with $k$ cycles is the
\emph{Narayana number}
\[
\mathrm{Nara}(n,k) = \frac{1}{n} \binom{n}{k-1} \binom{n}{k}.
\]

We denote by $\Par(n,k)$ the set of integer partitions of $n$ with $k$ parts.
If $\lambda$ has $m_i$ parts of size $i$ for $i=1,2,\dots$, then we also write
as $\lambda=(1^{m_1}, 2^{m_2},\dots)$. We will use the following notations: if
$\lambda=(1^{m_1}, 2^{m_2},\dots)$ has $k$ parts, i.e. $m_1+m_2+\cdots=k$, we
define
\[
\binom{k}{\lambda} = \frac{k!}{m_1!m_2!\cdots}, \qquad
\qbinom{k}{\lambda} = \frac{[k]_q!}{[m_1]_q! [m_2]_q!\cdots},
\]
where we use the standard notations
\[
[n]_q = \frac{1-q^n}{1-q}, \qquad [n]_q! = [1]_q [2]_q \cdots [n]_q,
\qquad \qbinom{n}{k} = \frac{[n]_q!}{[k]_q! [n-k]_q!}.
\]

The \emph{cycle type} of a permutation $\pi$ is the partition $\lambda=(1^{m_1},
2^{m_2},\dots)$, where $m_i$ is the number of cycles with $i$ elements.
Kreweras \cite{Kreweras} showed that the number of noncrossing permutations of
$[n]$ with cycle type $\lambda=(1^{m_1}, 2^{m_2},\dots)\in \Par(n,k)$ is equal
to the \emph{Kreweras number}
\[
\mathrm{Kre}(\lambda): = \frac{1}{k} \binom{n}{k-1} \binom{k}{\lambda}.
\]

Bessis and Reiner \cite[Theorem 6.2]{Bessis2011} showed that if $X$ is the set
of noncrossing partitions of $[n]$ with type $\lambda=(1^{m_1},
2^{m_2},\dots)\in \Par(n,k)$,
\[
X(q) = \frac{1}{[k]_q} \qbinom{n}{k-1} \qbinom{k}{\lambda},
\]
and $C$ is the cyclic group of rotations acting on $X$, then $(X,X(q),C)$
exhibits the cyclic sieving phenomenon, see Section~2 for the definition.

Reiner and Sommers \cite{ReinerSommers} defined the \emph{$q$-Kreweras number}
for $\lambda=(1^{m_1}, 2^{m_2},\dots)\in \Par(n,k)$ by
\[
\mathrm{Kre}_q(\lambda) = \frac{q^{(n+1)(n-k)-\tau(\lambda)}}{[k]_q} \qbinom{n}{k-1}
  \qbinom{k}{\lambda},
\]
where $\tau(\lambda)=\sum_{i\ge1}{\lambda_i' \lambda_{i+1}'}$, and the
\emph{$q$-Narayana number} by
\[
\mathrm{Nara}_q(n,k) = \frac{q^{(n-k)(n+1-k)}} {[n]_q} \qbinom{n}{k-1}
  \qbinom{n}{k}. 
\]
Reiner and Sommers \cite{ReinerSommers} showed that
\[
\sum_{\lambda\in \Par(n,k)} \mathrm{Kre}_q(\lambda) = \mathrm{Nara}_q(n,k)
\]
and
\[
\sum_{k=0}^n \mathrm{Nara}_q(n,k) = \mathrm{Cat}_q(n),
\]
where $\mathrm{Cat}_q(n)$ is the \emph{$q$-Catalan number} defined by 
\[
\mathrm{Cat}_q(n) = \frac{1}{[n+1]_q} \qbinom{2n}{n}.
\]

In this paper we prove analogous results for annular noncrossing permutations.
Annular noncrossing permutations are an annulus-analog of noncrossing
permutations. They were first introduced by Mingo and Nica \cite{Mingo2004} and
studied further in \cite{GNO, Nica2009}. Recently Kim, Seo, and Shin
\cite{KSS_ANP} used annular noncrossing permutations to give a combinatorial
proof of Goulden and Jackson's formula \cite{Goulden1997} for the number of
minimal transitive factorizations of a product of two cycles.

This paper is organized as follows.  In Section~\ref{sec:enum-annul-noncr} we
define annular noncrossing permutations and give an instance of the cyclic
sieving phenomenon on them.  In Section~\ref{sec:annular-q-kreweras} we define
annular $q$-Kreweras numbers, three types of annular $q$-Narayana numbers, and
annular $q$-Catalan numbers and show that a sum of annular $q$-Kreweras numbers
becomes an annular $q$-Narayana number, and a sum of annular $q$-Narayana
numbers becomes an annular $q$-Catalan numbers. We then show that these
polynomials give cyclic sieving phenomena. In
Section~\ref{sec:annul-noncr-match} we enumerate annular noncrossing matchings.

\section{Enumeration of annular noncrossing permutations}
\label{sec:enum-annul-noncr}

Let $n$ and $m$ be positive integers.  An \emph{$(n,m)$-annulus} is an annulus
in which $1,2,\dots,n$ are arranged in clockwise order on the exterior circle
and $n+1,n+2,\dots,n+m$ are arranged in counter-clockwise order on the interior
circle. 

Let $(a_1,\dots,a_k)$ be a cycle whose elements are contained in $[n+m]$.  We
will represent this cycle inside an $(n,m)$-annulus by drawing an arrow from
$a_i$ to $a_{i+1}$ for each $i=1,2,\dots k$, where $a_{k+1}=a_1$.  An
\emph{interior cycle} (respectively~\emph{exterior cycle}) is a cycle all of
whose elements are on the interior (respectively~exterior) circle. A
\emph{connected cycle} is a cycle which contains both an element in the interior
circle and an element in the exterior cycle.  Suppose
$\{e_1<e_2<\dots<e_u\}=\{a_1,\dots,a_k\} \cap [n]$ and
$\{i_1<i_2<\dots<i_v\}=\{a_1,\dots,a_k\} \cap \{n+1,n+2,\dots,n+m\}$.  Then we
say that the cycle $(a_1,\dots,a_k)$ is \emph{oriented clockwise} if we can
express $(a_1,\dots,a_k)=(e'_1,\dots,e'_u,i'_1,\dots,i'_v)$ for some integers
$e'_1,\dots,e'_u,i'_1,\dots,i'_v$ with $(e'_1,\dots,e'_u)=(e_1,\dots,e_u)$ and
$(i'_1,\dots,i'_v)=(i_1,\dots,i_v)$.  In this case we say that the cycle
$(a_1,\dots,a_k)$ is of \emph{size} $k$, of \emph{exterior size} $u$ and of
\emph{interior size} $v$.

\begin{figure}
  \centering
\begin{pspicture}(-6,-6)(6,6) 
\psset{linewidth=1.5pt, nodesep =2pt}
\pscircle[linecolor=gray](0,0){5}
\pscircle[linecolor=gray](0,0){2}
\multido{\n=90+-40,\i=1+1}{9}{\evput{5}{\n}{\i}{\i}}
\multido{\n=0+60,\i=10+1}{6}{\ivput{2}{\n}{\i}{\i}}
\edge 12 \edge 23 
\edge6{15} \ncarc[arcangle=-50]{->}{15}{10}
\edge6{15} \ncarc[arcangle=-50]{->}{10}{11}
\edge6{15} \ncarc[arcangle=-50]{->}{13}{14}
\edge{11}1
\ncarc[arcangle=10]{->}{3}{6}
\nccircle[angleA=40]{<-}{12}{4pt}
\ncarc[arcangle=0]{->}{4}{5}
\ncarc[arcangle=50]{->}{5}{4}
\edge{14}7 \edge78 \edge89 \edge9{13}
\end{pspicture}  
\caption{A representation of the annular noncrossing permutation
  $(1,2,3,6,15,10,11)(4,5)(7,8,9,13,14)(12)$.}
  \label{fig:ex}
\end{figure}
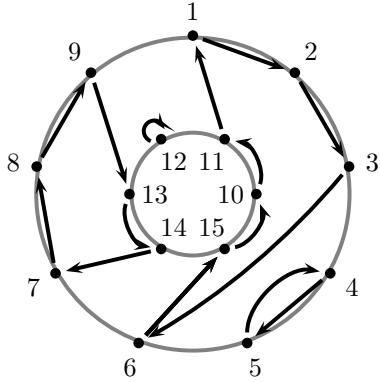

A permutation of $[n+m]$ is called an \emph{$(n,m)$-annular noncrossing
  permutation} if we can draw its cycles inside an $(n,m)$-annulus in such a way
that every cycle is oriented clockwise and there are no crossing arrows, see
Figure~\ref{fig:ex}. 

\begin{remark}
  Unlike noncrossing permutations, the map changing each cycle to a block is not
  a one-to-one correspondence from annular noncrossing permutations to annular
  noncrossing partitions. However, as is shown in
  \cite[Proposition~4.4]{Mingo2004}, if there are at least two connected cycles,
  then this map becomes a bijection. Thus every result in this paper on annular
  noncrossing permutations with at least two connected cycles works for annular
  noncrossing partitions.
\end{remark}

If an $(n,m)$-annular noncrossing permutation has a connected cycle, it is
called \emph{connected}. Since a disconnected annular noncrossing permutation is
essentially a disjoint union of two noncrossing permutations, in this paper we
will only consider connected annular noncrossing permutations.

We denote by $\anc(n,m)$ the set of connected $(n,m)$-annular noncrossing
permutations.  For $\pi\in\anc(n,m)$, the \emph{exterior cycle type}
(respectively~\emph{interior cycle type}) of $\pi$ is the partition
$(1^{m_1},2^{m_2},\dots)$ where $m_i$ is the number of exterior cycles
(respectively~interior cycles) of size $i$.  The \emph{connected exterior cycle type}
(respectively~\emph{connected interior cycle type}) of $\pi$ is the partition
$(1^{m_1},2^{m_2},\dots)$ where $m_i$ is the number of connected cycles of
exterior size (respectively~interior size) $i$.

For integers $n,m,c,r,s,R,S\ge0$ and $\alpha\in \Par(R,r)$,
$\beta\in \Par(S,s)$, $\lambda\in \Par(n-R,c)$, and $\mu\in \Par(m-S,c)$, we
define the following:
\begin{itemize}
\item $\anc(n,m;c)$ is the set of $\pi\in\anc(n,m)$ with $c$ connected cycles. 
\item $\anc(n,m;c,r,s)$ is the set of $\pi\in\anc(n,m;c)$ with $r$ exterior
  cycles and $s$ interior cycles.
\item $\anc(n,m;c,r,s,R,S)$ is the set of $\pi\in\anc(n,m;c,r,s)$ such that the
  total size of exterior cycles is $R$ and the total size of interior cycles is
  $S$.
\item $\anc(n,m;c,r,s,R,S;\alpha,\beta,\lambda,\mu)$ is the set of
  $\pi\in\anc(n,m)$ with exterior cycle type $\alpha\in \Par(R,r)$, interior
  cycle type $\beta\in \Par(S,s)$, connected exterior cycle type
  $\lambda\in \Par(n-R,c)$, and connected interior cycle type
  $\mu\in \Par(m-S,c)$.
\end{itemize}

\begin{defn} 
  Suppose a cyclic group $C$ of order $n$ acts on a finite set $X$.  Let $X(q)$
  be a polynomial in $q$.  We say that $(X,X(q),C)$ exhibits the \emph{cyclic
    sieving phenomenon} (CSP) if $X(w(c)) = |\{x\in X: c(x)=x\}|$ for all $c\in
  C$. Here, $w:C\to \mathbb C^\times$ is an embedding of $C$ into the
  multiplicative group $\mathbb C^\times$ of nonzero complex numbers sending a
  cyclic generator of $C$ to a primitive $n$th root of unity.
\end{defn}

The CSP was first introduced by Reiner, Stanton, and White \cite{CSP}.  Recently
many instances of the CSP have been found.  In \cite{SaganCSP}, Sagan gives a
nice survey on the CSP. 
 
Usually the CSP is defined with a polynomial $X(q)$ whose coefficients are
integers.  In this paper we do not assume any conditions on the coefficients of
$X(q)$. In fact we will see an instance of the CSP with $X(q)$ whose
coefficients are not integers. However, by the definition of the CSP, $X(w(c))$
is always a nonnegative integer.

The goal of this section is to find an instance of the CSP for $\anc(n,m)$. To
this end let $C_1\times C_2$ be the product of two cyclic groups $C_1$ acting on
the exterior circle and $C_2$ acting on the interior circle.  Then $C_1\times
C_2$ gives a bicyclic action on $\anc(n,m)$ by $(c_1,c_2)\pi =
c_1(c_2(\pi))=c_2(c_1(\pi))$. One may wonder if this bicyclic action gives a
``bicyclic sieving phenomenon'' as in \cite{BRS}. However, this is not the case
by the next proposition.

\begin{prop}
  Let $(c_1,c_2)\in C_1\times C_2$. Then $(c_1,c_2)$ has no fixed points in
  $\anc(n,m)$ unless $c_1$ and $c_2$ have the same order. 
\end{prop}
\begin{proof}
  Let $d_1$ and $d_2$ be the orders of $c_1$ and $c_2$ respectively. Suppose
  $\pi\in \anc(n,m)$ is a fixed point, i.e. $(c_1,c_2)\pi = \pi$.  Let
  $k_1=n/d_1$ and $k_2=m/d_2$.  We can assume that $c_1$ is the map sending
  $i\in [n]$ to $j\in [n]$ with $j\equiv i+k_1 \bmod n$, and $c_2$ is the map
  sending $n+i\in \{n+1,\dots,n+m\}$ to $n+j\in \{n+1,\dots,n+m\}$ with $j\equiv
  i-k_2 \bmod m$. Note that for each $i\in[n]$ (respectively~$i\in\{n+1,\dots,n+m\}$)
  $c_1^{(t)}(i) = i$ (respectively~$c_2^{(t)}(i) = i$) implies that $t$ divides $d_1$
  (respectively~$d_2$).  Consider a connected cycle
  $\gamma=(a_1,\dots,a_u,b_1,\dots,b_v)$ of $\pi$ where $a_1,\dots,a_u\in [n]$
  and $b_1,\dots,b_v\in \{n+1,\dots,n+m\}$. Note that $u>0$ and $v>0$.  Since
  $(c_1,c_2)^{(d_1)}(\pi) = \pi$ and
  $(c_1,c_2)^{(d_1)}(\gamma)=(a_1,\dots,a_u,c_2^{(d_1)}(b_1) ,\dots,
  c_2^{(d_1)}(b_v))$, we have $(c_1,c_2)^{(d_1)}(\gamma) = \gamma$. In
  particular, $c_2^{(d_1)}(b_1) = b_1$, which implies that $d_1$ divides
  $d_2$. Similarly we get that $d_2$ divides $d_1$. Thus $d_1=d_2$.
\end{proof}

Thus we will only consider the elements $(c_1,c_2)$ for which $c_1$ and $c_2$
have the same order. We call such pair $(c_1,c_2)$ an \emph{$(n,m)$-annular
  rotation}, or simply an \emph{annular rotation}.  Note that if $(c_1,c_2)$ is
an $(n,m)$-annular rotation, then the order $d$ of this action divides both $n$
and $m$.

For a partition $\lambda=(1^{m_1},2^{m_2},\dots)$ we denote by $\mathfrak
S(\lambda)$ the set of rearrangements of the following sequence
\[
\overbrace{1,\dots,1}^{m_1},\overbrace{2,\dots,2}^{m_2},\dots.
\]
 In other words,
each element in $\mathfrak S(\lambda)$ is a sequence $(a_1,a_2,\dots)$ where
each integer $i$ appears exactly $m_i$ times.  If each $m_i$ is divisible by
$d$, we define $\lambda/d=(1^{m_1/d}, 2^{m_2/d},\dots)$. In this case we say
that $\lambda$ is \emph{divisible} by $d$.

\begin{thm}\label{thm:rot_inv}
  The number of $\pi\in \anc(n,m;c,r,s,R,S;\alpha,\beta,\lambda,\mu)$ invariant
  under an annular rotation of order $d$ is equal to
\[
d\cdot \frac{(\widehat{n}-\widehat{R})(\widehat{m}-\widehat{S})}{\widehat{c}}
\binom{\widehat{n}}{\widehat{r}}\binom{\widehat{m}}{\widehat{s}}
\binom{\widehat{r}}{\widehat{\alpha}} \binom{\widehat{s}}{\widehat{\beta}}
\binom{\widehat{c}}{\widehat{\lambda}} \binom{\widehat{c}}{\widehat{\mu}},
\]
if all of $n,m,c,r,s,R,S,\alpha,\beta,\lambda,\mu$ are divisible by $d$, and $0$
otherwise.  Here $\widehat Z$ means $Z/d$.
\end{thm}
\begin{proof}[Sketch of Proof]
  Since the proof is similar to those in \cite[Proposition~4.2]{GNO} and in
  \cite[Proposition~4.1]{Kim}, we will only give a sketch. 

  If $d=1$, this follows from Theorem~\ref{thm:rot_inv}. Suppose $d>1$.  It is
  not difficult to see that if an annular noncrossing permutation $\pi$ is
  invariant under an annular rotation of order $d$, then $\pi$ has no cycles
  invariant under this rotation. Thus if any of
  $n,m,c,r,s,R,S,\alpha,\beta,\lambda,\mu$ is not divisible by $d$, then there
  is no such $\pi$.  So we will assume that all of
  $n,m,c,r,s,R,S,\alpha,\beta,\lambda,\mu$ are divisible by $d$.

  Let $(c_1,c_2)$ be an annular rotation of order $d$.  We will find a bijection
  between the set
\[
\A=\{(\gamma,\pi): \pi\in \anc(n,m;c,r,s,R,S;\alpha,\beta,\lambda,\mu), 
(c_1,c_2)\pi = \pi, \mbox{$\gamma$ is a connected cycle of $\pi$}\}
\]
and the set
\[
\B = \left\{(a,b,R^E,R^I,V^E,V^I,V^{CE},V^{CI}) \left| 
\begin{array}{l}
a\in[n - R], 
b\in[m - S], \\
R^E\subset[\widehat n],
R^I\subset[\widehat m],
|R^E|=\widehat r, |R^I|=\widehat s,
\\
V^E\in \mathfrak S(\widehat \alpha), 
V^I\in \mathfrak S(\widehat \beta), 
V^{CE}\in \mathfrak S(\widehat \lambda), 
V^{CI}\in \mathfrak S(\widehat \mu)
\end{array}
\right.
\right\}.
\]

Suppose $(\gamma,\pi)\in \A$. Let $\gamma=(a_1,\dots,a_u,b_1,\dots,b_v)$ with
$a_1,\dots,a_u\in [n]$ and $b_1,\dots,b_v\in \{n+1,\dots,n+m\}$.  We define
$a=a_1$ and $b=b_v$. Arrange the integers in $[n+m]$ as follows:
\begin{equation}
  \label{eq:1}
a, a+1, \dots, n, 1, 2, \dots, a-1, b+1, b+2, \dots, n+m, n+1, n+2,\dots, b.  
\end{equation}
For each exterior or interior cycle of size $t$, we place a right parenthesis
$)_t$ labeled by $t$ after the rightmost integer in the sequence \eqref{eq:1}
which is an element of the cycle. We define $R^E$ (respectively~$R^I$) to be the set of
integers $i$ with a right parenthesis such that $i\in [\widehat n]$
(respectively~$i-n\in [\widehat m]$). Then we have $|R^E|=\widehat r$ and
$|R^I|=\widehat s$. Let $i_1<i_2<\dots<i_{\widehat r}$ be the elements of $R^E$.
We define $V^E$ to be the sequence $(\ell_1,\ell_2,\dots,\ell_{\widehat r})$
where $\ell_j$ is the label of the right parenthesis after $i_j$.  The sequence
$V^I$ is defined similarly.

Now remove the integers contained in an exterior or an interior cycle from the
sequence \eqref{eq:1}.  For each connected cycle, we place a left
(respectively~right) parenthesis before (respectively~after) the leftmost
(respectively~rightmost) integer in the remaining sequence which is an element
of the cycle. Then the $c$ left parentheses divide the first part of the
remaining sequence consisting of integers at most $n$ into $c$ subsequences.
Let $V^{CE}$ be the sequence of sizes of the $c$ subsequences.  Similarly, the
$c$ right parentheses divide the second part of the remaining sequence
consisting of integers greater than $n$ into $c$ subsequences. We define
$V^{CI}$ to be the sequence of sizes of the $c$ subsequences.  We have just
constructed the map $(\gamma,\pi)\mapsto(a,b,R^E,R^I,V^E,V^I,V^{CE},V^{CI})$.
Using the ideas of \cite[Proposition~4.2]{GNO} and \cite[Proposition~4.1]{Kim} ,
one can show that this gives a bijection from $\A$ to $\B$.

Thus the number of $\pi\in \anc(n,m;c,r,s,R,S;\alpha,\beta,\lambda,\mu)$
invariant under an annular rotation of order $d$ is $|\A|/c =|\B|/c$, which is
easily seen to equal to the number in the theorem.
\end{proof}

Now we state the main theorem of this section. 

\begin{thm}\label{thm:CSP}
  Suppose that the cyclic group $C$ of $(n,m)$-annular rotations acts on the set
 \[
X=\anc(n,m;c,r,s,R,S;\alpha,\beta,\lambda,\mu).
\] Let
\[
X(q) = \frac{[(n-R)(m-S)]_q}{[c]_q} \qbinom{n}{r}\qbinom{m}{s}
\qbinom{r}{\alpha} \qbinom{s}{\beta} \qbinom{c}{\lambda} \qbinom{c}{\mu}.
\]
Then $(X, X(q), C)$ exhibits the CSP.
\end{thm}

We first prove that the $X(q)$ in Theorem~\ref{thm:CSP} is a polynomial in
$q$. We need two lemmas.

\begin{lem}\cite[Proposition~10.1 (iii)]{CSP}\label{lem:RSW}
  If $f(q) = h(q)/[k]_q \in\Z[q]$, and $h(q)\in \N[q]$ has symmetric, unimodal
  coefficient sequence, then $f(q)\in \N[q]$.
\end{lem}

\begin{lem}\label{lem:polynomial}
  For nonnegative integers $N,n,k$ with $N\ge n$ and a partition
  $\lambda\in \Par(n,k)$, we have
\[
\frac{[n]_q}{[k]_q} \qbinom{k}{\lambda}\in \N[q], \qquad
\frac{[N-n]_q}{[N]_q} \qbinom{N}{k} \qbinom{k}{\lambda}\in \N[q].
\]
\end{lem}
\begin{proof}
We claim that
\begin{equation}
  \label{eq:3}
\frac{[n]_q}{[k]_q} \qbinom{k}{\lambda}
=\frac{1-q^n}{1-q^k} \qbinom{k}{\lambda}\in \Z[q].
\end{equation}
Assuming the claim let us show the lemma.  By \eqref{eq:3} and the fact
$[N]_q=[N-n]_q+q^{N-n}[n]_q$ we also have
\begin{align*}
\frac{[N-n]_q}{[N]_q} \qbinom{N}{k} \qbinom{k}{\lambda}
&=\frac{[N]_q-q^{N-n}[n]_q}{[N]_q} \qbinom{N}{k} \qbinom{k}{\lambda}\\
&=\qbinom{N}{k} \qbinom{k}{\lambda} - 
q^{N-n} \qbinom{N-1}{k-1} \cdot 
\frac{[n]_q}{[k]_q} \qbinom{k}{\lambda} \in \Z[q].
\end{align*}
Since both $\frac{[n]_q}{[k]_q} \qbinom{k}{\lambda}$ and $\frac{[N-n]_q}{[N]_q}
\qbinom{N}{k} \qbinom{k}{\lambda}$ satisfy the conditions in
Lemma~\ref{lem:RSW} we obtain the lemma.

We now show \eqref{eq:3}. We will use the $q$-Pochhammer symbol $(q;q)_r =
(1-q)(1-q^2)\cdots(1-q^r)$. Note that $\qbinom{n}{k} = \frac{(q;q)_n}{(q;q)_k
  (q;q)_{n-k}}$.

 Since
\[
q^k-1 = \prod_{j=1}^k (q-\omega_j),
\]
where $\omega_1,\omega_2,\dots,\omega_k$ are the $k$th roots of unity, in order
to prove \eqref{eq:3} it is sufficient to show that $q-\omega_j$ divides
$(1-q^n) \qbinom{k}{\lambda}$ for all $j=1,2,\dots,k$. Fix an integer $j$ and
suppose $\omega_j$ is a primitive $r$th root of unity. Then $r$ divides
$k$. Note that $q-\omega_j$ divides $q^s-1$ if and only if $r$ divides $s$.
Note also that the multiplicity of $q-\omega_j$ as a factor of $q^s-1$ is at
most $1$. Thus the multiplicity of the factor $q-\omega_j$ in $(q;q)_s$ is equal
to $\flr{s/r}$.  We have two cases as follows. 

\textsc{Case 1}: $r$ divides $n$. Then $q-\omega_j$ divides $q^n-1$ and we
are done. 

\textsc{Case 2}: $r$ does not divide $n$.  Let $\lambda=(1^{m_1},2^{m_2},\dots,
\ell^{m_{\ell}})$. Then we have $n = \sum_{i=1}^\ell i \cdot m_i$, $k =
\sum_{i=1}^\ell m_i$, and 
\begin{equation}
  \label{eq:4}
\qbinom{k}{\lambda} = \frac{(q;q)_{m_1+m_2+\cdots+m_\ell}}
{(q;q)_{m_1} (q;q)_{m_2} \cdots (q;q)_{m_\ell}}. 
\end{equation}
The multiplicities of the factor $q-\omega_j$ in the numerator and denominator
of \eqref{eq:4} are respectively $\flr{k/r}$ and $\flr{m_1/r} + \flr{m_2/r} +
\cdots + \flr{m_\ell/r}$. Since $r$ does not divide $n$, at least one of
$m_1,m_2,\dots, m_\ell$ is not a multiple of $r$. Thus we have
\[
\flr{\frac{m_1}{r}} + \flr{\frac{m_2}r} + \cdots + \flr{\frac{m_\ell}r}
< \frac{m_1}{r} + \frac{m_2}r + \cdots + \frac{m_\ell}r = 
\frac{k}{r} = \flr{\frac{k}{r}},
\]
which implies that $q-\omega_j$ divides $\qbinom{k}{\lambda}$. This finishes the
proof of \eqref{eq:3}.
\end{proof}

\begin{prop}
  The $X(q)$ in Theorem~\ref{thm:CSP} is a polynomial in $q$.
\end{prop}
\begin{proof}
By Lemma~\ref{lem:polynomial}, $\frac{[m-S]_q}{[c]_q} \qbinom{c}{\mu}$ is a
polynomial. Since
\begin{equation}
  \label{eq:11}
X(q) = 
\frac{[(n-R)(m-S)]_q}{[m-S]_q} \qbinom{n}{r}\qbinom{m}{s}
\qbinom{r}{\alpha} \qbinom{s}{\beta} \qbinom{c}{\lambda} 
\frac{[m-S]_q}{[c]_q} \qbinom{c}{\mu},
\end{equation}
and $[ab]_q = [a]_q [b]_{q^a} = [b]_q[a]_{q^b}$, we obtain that $X(q)$ is a
polynomial in $q$.
\end{proof}

The following evaluations of $q$-binomial coefficients at a root of unity are
well known, for instance see \cite[Exercise~96 in Chapter 1]{EC1}.

\begin{lem} \label{lem:eval}
  Suppose $w$ is a primitive $d$th root of unity and $n$ is divisible by
  $d$. Then
$[n]_{q=w} = n/d$ and 
\[
\Qbinom{n}{k}{q=w} = \left\{
  \begin{array}{ll}
    \binom{n/d}{k/d} & \mbox{if $k$ is divisible
      by $d$,}\\
    0 & \mbox{otherwise.}
  \end{array}\right.
\]
\end{lem}

Now we are ready to prove Theorem~\ref{thm:CSP}. 

\begin{proof}[Proof of Theorem~\ref{thm:CSP}]
  Suppose $(c_1,c_2)$ is an $(n,m)$-annular rotation of order $d$. Then $d$
  divides both $n$ and $m$.  By Theorem~\ref{thm:rot_inv} it is sufficient to
  show that for $w$ a primitive $d$th root of unity, we have
  \begin{equation}
    \label{eq:9}
X(w) =     d\cdot \frac{(\widehat{n}-\widehat{R})(\widehat{m}-\widehat{S})}{\widehat{c}}
\binom{\widehat{n}}{\widehat{r}}\binom{\widehat{m}}{\widehat{s}}
\binom{\widehat{r}}{\widehat{\alpha}} \binom{\widehat{s}}{\widehat{\beta}}
\binom{\widehat{c}}{\widehat{\lambda}} \binom{\widehat{c}}{\widehat{\mu}},
  \end{equation}
  if all of $n,m,c,r,s,R,S,\alpha,\beta,\lambda,\mu$ are divisible by $d$, and
  $X(w)=0$ otherwise. By Lemma~\ref{lem:eval} we get~\eqref{eq:9} when
  all of $n,m,c,r,s,R,S,\alpha,\beta,\lambda,\mu$ are divisible by $d$.  

  It remains to show that if at least one of
  $n,m,c,r,s,R,S,\alpha,\beta,\lambda,\mu$ is not divisible by $d$, then
  $X(w)=0$.  Since $d$ divides both $n$ and $m$, we have the following cases.

  \textsc{Case 1}: $r$ or $s$ is not divisible by $d$. By \eqref{eq:11}, $X(q)$
  is a polynomial divisible by $\qbinom{n}{r}\qbinom{m}{s}$. Thus, by
  Lemma~\ref{lem:eval}, we get $X(w)=0$. 

  \textsc{Case 2}: Both $r$ and $s$ are divisible by $d$, but $\alpha$ or
  $\beta$ is not. Suppose $\alpha=(1^{a_1},2^{a_2},\dots)$ is not divisible by
  $d$. Suppose moreover that $a_j$ is not divisible by $d$. Again by
  \eqref{eq:11}, $X(q)$ is divisible by 
\[
\qbinom{r}{\alpha} = \qbinom{r}{a_j}\qbinom{r-a_j}{a_1,\dots,a_{j-1},a_{j+1},\dots}.
\]
Thus, by Lemma~\ref{lem:eval}, we get $X(w)=0$. If $\beta$ is not divisible by
$d$, by the same arguments we get $X(w)=0$. 

\textsc{Case 3}: All of $r,s,\alpha,\beta$ are divisible by $d$, but $c$ is
not. Note that since $\alpha$ and $\beta$ are divisible by $d$, so are $R$ and
$S$. By \eqref{eq:11}, we can write
\[
X(q)=\frac{1-q^{m-S}}{1-q^c} Y(q)
\]
for a polynomial $Y(q)$. Since $m-S$ is divisible by $d$, but $c$ is not,
$\frac{1-w^{m-S}}{1-w^c}=0$ and we get $X(w)=0$.

\textsc{Case 4}: All of $r,s,\alpha,\beta,R,S,c$ are divisible by $d$, but
$\lambda$ or $\mu$ is not. Suppose $\lambda$ is not divisible by $d$. By
\eqref{eq:11}, $X(q)$ is divisible by $\qbinom{c}{\lambda}$. By the same
argument as in \textsc{Case 2} we obtain that $X(w)=0$. If $\mu$ is not
divisible by $d$, we can do the same thing using the following expression 
\[
X(q) = 
\frac{[(n-R)(m-S)]_q}{[n-R]_q} \qbinom{n}{r}\qbinom{m}{s}
\qbinom{r}{\alpha} \qbinom{s}{\beta} \qbinom{c}{\mu}
\frac{[n-R]_q}{[c]_q}\qbinom{c}{\lambda} .
\]

Thus in all cases we have $X(w)=0$, which finishes the proof. 
\end{proof}

\section{Annular $q$-Kreweras numbers}
\label{sec:annular-q-kreweras}

In this section we define annular versions of $q$-analogs of Kreweras, Narayana,
and Catalan numbers and evaluate their sums.

For brevity we will use the following throughout this section:
\begin{align*}
  X &= {c(c-1)},  \\
  Y &= {r(c+r)+s(c+s)},  \\
  Z &= {r(n-c-R)+s(m-c-S)},  \\
  W &=
  {r(R-r)+s(S-s)+c(n-R-c)+c(m-S-c)-\tau(\alpha)-\tau(\beta)-\tau(\lambda)-\tau(\mu)}.
\end{align*}

\begin{defn}
The \emph{annular $q$-Kreweras number
    $\Kre=\Kre(n,m;c,r,s,R,S;\alpha,\beta,\lambda,\mu)$} is defined by
\[
\Kre = q^X q^Y q^Z q^W \frac{[nm]_q}{[n]_q[m]_q} \frac{[2c]_q}{2} 
\frac{[n-R]_q[m-S]_q}{[c]_q^2}  \qbinom{n}{r}\qbinom{m}{s} 
\qbinom{r}{\alpha} \qbinom{s}{\beta} \qbinom{c}{\lambda} \qbinom{c}{\mu}.
\]
The \emph{annular $q$-Narayana number $\Nara_1=\Nara_1(n,m;c,r,s,R,S)$ of type
  $1$} is defined by
\[
\Nara_1 =  q^X q^Y q^Z \frac{[nm]_q}{[n]_q[m]_q} \frac{[2c]_{q}}{2}
\qbinom{n}{r}\qbinom{m}{s}
\qbinom{R-1}{r-1} \qbinom{S-1}{s-1} \qbinom{n-R}{c} \qbinom{m-S}{c}.
\]
The \emph{annular $q$-Narayana number $\Nara_2=\Nara_2(n,m;c,r,s)$ of type $2$}
is defined by
\[
\Nara_2 = q^X q^Y  \frac{[nm]_q}{[n]_q[m]_q} \frac{[2c]_{q}}{2} 
\qbinom{n}{r}\qbinom{m}{s} \qbinom{n}{r+c} \qbinom{m}{s+c}.
\]
The \emph{annular $q$-Narayana number
    $\Nara_3=\Nara_3(n,m;,c)$ of type $3$} is defined by
\[
\Nara_3 = q^X \frac{[nm]_q}{[n]_q[m]_q} \frac{[2c]_{q}}{2} 
\qbinom{2n}{n-c} \qbinom{2m}{m-c}.
\]
The \emph{annular $q$-Catalan number $\Cat(n,m)$} is defined by
\[
\Cat(n,m) = \frac{[nm]_q}{2[m+n]_q} \qbinom{2n}{n} \qbinom{2m}{m}.
\]
\end{defn}

In what follows we show that the above numbers are polynomials in $q$. Note
however that they are not necessarily polynomials of integer coefficients. For
instance, $\Cat(1,1) = (1+q)/2$.  We first show that $\Kre$ is a polynomial in
$q$.

\begin{prop}
  The annular $q$-Kreweras number
  $\Kre=\Kre(n,m;c,r,s,R,S;\alpha,\beta,\lambda,\mu)$ is a polynomial in $q$.
\end{prop}
\begin{proof}
Since
\[
\Kre = q^X q^Y q^Z q^W \frac{[2c]_q}{2[c]_q} 
\left( \frac{[n-R]_q}{[n]_q} \qbinom{n}{r}\qbinom{r}{\alpha} \right)
\left( \frac{[m-S]_q}{[c]_q} \qbinom{c}{\mu} \right)
\left( \frac{[nm]_q}{[m]_q} \qbinom{m}{s}
\qbinom{s}{\beta} \qbinom{c}{\lambda} \right),
\]
we are done by Lemma~\ref{lem:polynomial}. 
\end{proof}

In the introduction we saw that the sum of $q$-Kreweras numbers is equal to the
$q$-Narayana number, and the sum of $q$-Narayana numbers is equal to the
$q$-Catalan number. We show that annular versions of these numbers have similar
properties. In order to do this we prove three lemmas.

The following lemma is due to Reiner and Sommers \cite{ReinerSommers}.  We
include their elegant proof as well.
\begin{lem}\cite{ReinerSommers}\label{lem:sum2}
  Let $\tau(\lambda) = \sum_{i\ge1} \lambda_i' \lambda_{i+1}'$, where $\lambda'$
  is the transposition of $\lambda$. Then
\[
\sum_{\lambda\in \Par(n,k)}q^{k(n-k)-\tau(\lambda)} \qbinom{k}{\lambda}
=\qbinom{n-1}{k-1}.
\]
\end{lem}
\begin{proof}
Let $\lambda=(1^{m_1},2^{m_2},\dots)$. Then $m_1+2m_2+\dots=n$ and 
$m_1+m_2+\dots=k$. Since $\lambda'_i=m_i+m_{i+1}+\cdots$, we have
\[
q^{k(n-k)-\tau(\lambda)} \qbinom{k}{\lambda}
= q^{(k-\lambda'_1)\lambda'_{2}} \qbinom{\lambda'_1}{\lambda'_2}
q^{(k-\lambda'_2)\lambda'_{3}} \qbinom{\lambda'_2}{\lambda'_3} \cdots . 
\]
Thus we can rewrite the identity as follows.
\begin{equation}
  \label{eq:2}
\qbinom{n-1}{k-1} = \sum_{\mu\in \Par(n), \mu_1=k} 
q^{(k-\mu_1)\mu_{2}} \qbinom{\mu_1}{\mu_2}
q^{(k-\mu_2)\mu_{3}} \qbinom{\mu_2}{\mu_3} \cdots . 
\end{equation}
The left hand side of \eqref{eq:2} is the sum of $q^{|\nu|}$ for all partitions
$\nu$ contained in a $k\times(n-k)$ rectangle where the $k$th part of $\nu$ is
$0$.  For such a partition $\nu$ we define points $Q_1,P_1,Q_2,P_2,\dots$ as
follows.  Let $Q_1$ be the upper right corner of the rectangle. When $Q_i$ is
defined, let $P_i$ be the point on the base of the rectangle which is straightly
below $Q_i$. When $P_i$ is defined, let $Q_{i+1}$ be the intersection of the
border of $\nu$ and the northwest diagonal ray starting from $P_i$. We define
the sequence of points until we reach the bottom-left corner of the
rectangle. Since $\nu$ has no cells in the $k$th row of the rectangle, we can
always complete this sequence.  Let $\mu$ be the partition whose $i$th part is
equal to the length of the segment $P_{i-1}P_i$, where $P_0=Q_1$. Then
$\mu\in \Par(n)$ and $\mu_1=k$.  It is easy to see that the sum of $q^{|\nu|}$
for partitions $\nu$ which give $\mu$ is equal to the summand in \eqref{eq:2},
see Figure~\ref{fig:decomp}.  This proves \eqref{eq:2}.
\end{proof}

\begin{figure}
  \centering
\psset{unit=25pt}
\begin{pspicture}(0,-1)(10,9)
\psframe(0,0)(10,8)
\psline[linestyle=dotted](0,2)(2,0)
\psline[linestyle=dotted](2,3)(5,0)
\psline[linestyle=dotted](5,5)(10,0)
\psline[linestyle=dotted](5,0)(5,8)
\psline[linestyle=dotted](2,5)(10,5)
\psline[linestyle=dotted](2,0)(2,8)
\psline[linestyle=dotted](0,2)(2,2)
\psline[linestyle=dotted](0,3)(5,3)
\psline[linewidth=2pt](0,0)(0,2)(1,2)(1,2.5)(2,2.5)(2,3.5)(3,3.5)(3,4)(4,4)(4,5)
(6,5)(6,6)(9,6)(9,7)(10,7)(10,8)
\rput(5,8.5){\psscaleboxto(9.8,.5){\rotateright{\{}}} \rput(5,9){$n-k$}
\rput(10.5,4){\psscaleboxto(.5,7.8){\}}} \rput[l](11,4){$k=\mu_1$}
\rput(1,-.5){\psscaleboxto(1.8,-.5){\rotateright{\{}}} \rput(1,-1){$\mu_4$}
\rput(3.5,-.5){\psscaleboxto(2.8,-.5){\rotateright{\{}}} \rput(3.5,-1){$\mu_3$}
\rput(7.5,-.5){\psscaleboxto(4.8,-.5){\rotateright{\{}}} \rput(7.5,-1){$\mu_2$}
\rput(.5,.5){$P_4$} \rput(2.5,.5){$P_3$} \rput(5.5,.5){$P_2$}
\rput(10.5,-.5){$P_1$}
\rput(-.5,2){$Q_4$} \rput(1.5,3.5){$Q_3$}
\rput(4.5,5.5){$Q_2$}  \rput[lt](10.2,8.5){$Q_1=P_0$}
\rput(1,5.5){$q^{(k-\mu_3)\mu_4}$}
\rput(3.5,6.5){$q^{(k-\mu_2)\mu_3}$}
\rput(7,7){$\qbinom{\mu_1}{\mu_2}$}
\rput(3,4.5){$\qbinom{\mu_2}{\mu_3}$}
\rput(.5,2.5){$\qbinom{\mu_3}{\mu_4}$}
\end{pspicture}
  \caption{We decompose $\nu$ and get $\mu$.}
  \label{fig:decomp}
\end{figure}

\begin{lem}\label{lem:sum3}
For fixed integers $n,r,c$, we have
\[
\sum_{R\ge0} q^{r(n-c-R)} \qbinom{R-1}{r-1} \qbinom{n-R}{c}
= \qbinom{n}{r+c}.
\]
\end{lem}
\begin{proof}
  This can be proved by a standard technique considering the largest rectangle
  with width $r$ contained in partitions inside an $(n-r-c)\times(r+c)$
  rectangle.
\end{proof}

\begin{lem}\label{lem:sum1}
Let $n$, $m$, and $k$ be any nonnegative integers. Then 
\[
\sum_{c\ge0} q^{c(c-1+k)} [2c+k]_q 
\qbinom{2n+k}{n-c}\qbinom{2m+k}{m-c}
=\frac{[n+k]_q [m+k]_q}{[n+m+k]_q}
\qbinom{2n+k}{n+k} \qbinom{2m+k}{m+k}.
\]
\end{lem}

\begin{proof}
  It is straightforward to check that
\begin{equation}
  \label{eq:5}
q^{c(c-1+k)} [n+m+k]_q [2c+k]_q =
q^{c(c-1+k)} [n+c+k]_q [m+c+k]_q - q^{(c+1)(c+k)}[n-c]_q[m-c]_q. 
\end{equation}
Let $\sigma(c) = c(c-1+k)$. Then by \eqref{eq:5} the left hand side is equal
to
\begin{align*}
&\frac{1}{[n+m+k]_q} \sum_{c\ge0} 
\big( 
q^{\sigma(c)} [n+c+k]_q [m+c+k]_q 
- q^{\sigma(c+1)}[n-c]_q[m-c]_q \big)
\qbinom{2n+k}{n-c}\qbinom{2m+k}{m-c}\\
&= \frac{[2n+k]_q [2m+k]_q}{[n+m+k]_q} 
\sum_{c\ge0} \left( 
q^{\sigma(c)} 
\qbinom{2n+k-1}{n-c} \qbinom{2m+k-1}{m-c} 
- q^{\sigma(c+1)}
\qbinom{2n+k-1}{n-c-1} \qbinom{2m+k-1}{m-c-1}
\right)\\
&= \frac{[2n+k]_q [2m+k]_q}{[n+m+k]_q} 
\qbinom{2n+k-1}{n} \qbinom{2m+k-1}{m}\\
&=\frac{[n+k]_q [m+k]_q}{[n+m+k]_q}
\qbinom{2n+k}{n+k} \qbinom{2m+k}{m+k}. 
\end{align*}
\end{proof}

\begin{thm}\label{thm:kreweras sum}
We have
\begin{align*}
  \sum_{\substack{\alpha\in \Par(R,r)\\
\beta\in \Par(S,s)\\
\lambda\in \Par(n-R,c)\\
\mu\in \Par(m-S,c)}}
\Kre(n,m;c,r,s,R,S;\alpha,\beta,\lambda,\mu)
& = \Nara_1(n,m;c,r,s,R,S),\\
\sum_{R,S\ge0} \Nara_1(n,m;c,r,s,R,S) &= \Nara_2(n,m;c,r,s),\\
\sum_{r,s\ge0} \Nara_2(n,m;c,r,s) &= \Nara_3(n,m;c),\\
\sum_{c\ge0} \Nara_3(n,m;c) &= \Cat(n,m).
\end{align*}
\end{thm}
\begin{proof}
  The first, second, and fourth identities follow from Lemmas~\ref{lem:sum2},
  \ref{lem:sum3}, and \ref{lem:sum1}, respectively.  The third identity follows
  from the $q$-Vandermonde's identity:
\[
\sum_{i\ge0} q^{i(m-k+i)} \qbinom{m}{k-i} \qbinom{n}{i} 
=\qbinom{m+n}{k}.
\]
\end{proof}

It is easy to see that we can replace the polynomial $X(q)$ in
Theorem~\ref{thm:CSP} with $\Kre$. Thus by Theorems~\ref{thm:CSP} and
\ref{thm:kreweras sum} we obtain the following.

\begin{thm}\label{thm:mcsp}
The following exhibit the cyclic sieving phenomenon:
\[
(\anc(n,m;c,r,s,R,S;\alpha,\beta,\lambda,\mu), C,
\Kre(n,m;c,r,s,R,S;\alpha,\beta,\lambda,\mu)),
\]
\[
(\anc(n,m;c,r,s,R,S), C, \Nara_1(n,m;c,r,s,R,S)),
\]
\[
(\anc(n,m;c,r,s), C, \Nara_2(n,m;c,r,s)),
\]
\[
(\anc(n,m;c), C, \Nara_3(n,m;c)),
\]
\[
(\anc(n,m), C, \Cat(n,m)).
\]
\end{thm}

Considering the annular rotation of order $1$ in Theorem~\ref{thm:mcsp},
i.e. the identity action, we obtain the following.

\begin{cor}
We have
\begin{align}
\label{eq:anc}
\#\anc(n,m;c,r,s,R,S;\alpha,\beta,\lambda,\mu)
&=\frac{(n-R)(m-S)}{c}  \binom{n}{r}\binom{m}{s} 
\binom{r}{\alpha} \binom{s}{\beta} \binom{c}{\lambda} \binom{c}{\mu},\\
\#\anc(n,m;c,r,s,R,S)
&=c \binom{n}{r}\binom{m}{s}
\binom{R-1}{r-1} \binom{S-1}{s-1} \binom{n-R}{c} \binom{m-S}{c},\\
\#\anc(n,m;c,r,s)
&= c \binom{n}{r}\binom{m}{s} \binom{n}{r+c} \binom{m}{s+c},\\
\#\anc(n,m;c)
&= c \binom{2n}{n-c} \binom{2m}{m-c},\\
\label{eq:anc_cat}
\#\anc(n,m)
&= \frac{nm}{2(m+n)} \binom{2n}{n} \binom{2m}{m}.
\end{align}
\end{cor}

If an annular noncrossing permutation is invariant under an annular rotation of
order $2$, it is called an \emph{annular noncrossing permutation of type $B$}.
We define $\anc_B(n,m)$ to be the set of connected $(2n,2m)$-annular noncrossing
permutations of type $B$.  We then define
\begin{align*}
\anc_B(n,m;c) &= \anc_B(n,m) \cap \anc(2n,2m;2c),\\
\anc_B(n,m;c,r,s) &= \anc_B(n,m) \cap \anc(2n,2m;2c,2r,2s),\\
\anc_B(n,m;c,r,s,R,S) &= \anc_B(n,m) \cap \anc(2n,2m;2c,2r,2s,2R,2S),\\
\anc_B(n,m;c,r,s,R,S;\alpha,\beta,\lambda,\mu)
&=\anc_B(n,m) \cap \anc(2n,2m;2c,2r,2s,2R,2S;2\alpha,2\beta,2\lambda,2\mu),
\end{align*}
where $2\lambda = (2\lambda_1,2\lambda_2,\dots)$.

Notice that every connected annular noncrossing permutation of type $B$ contains
at least two connected cycles.  Thus annular noncrossing permutations of type
$B$ are in bijection with annular noncrossing partitions of type $B$.

Considering an annular rotation of order $2$ in Theorem~\ref{thm:mcsp} we obtain
the following.

\begin{cor}
We have
\begin{align}
\#\anc_B(n,m;c,r,s,R,S;\alpha,\beta,\lambda,\mu)
&=\frac{2(n-R)(m-S)}{c}  \binom{n}{r}\binom{m}{s} 
\binom{r}{\alpha} \binom{s}{\beta} \binom{c}{\lambda} \binom{c}{\mu},\\
\#\anc_B(n,m;c,r,s,R,S)
&=2c \binom{n}{r}\binom{m}{s}
\binom{R-1}{r-1} \binom{S-1}{s-1} \binom{n-R}{c} \binom{m-S}{c},\\
\label{eq:b1}
\#\anc_B(n,m;c,r,s)
&= 2c \binom{n}{r}\binom{m}{s} \binom{n}{r+c} \binom{m}{s+c},\\
\label{eq:b2}
\#\anc_B(n,m;c)
&= 2c \binom{2n}{n-c} \binom{2m}{m-c},\\
\label{eq:b3}
\#\anc_B(n,m)
&= \frac{nm}{m+n} \binom{2n}{n} \binom{2m}{m}.
\end{align}
\end{cor}

We note that \eqref{eq:anc_cat} was first proved by Mingo and Nica
\cite[Corollary~6.7]{Mingo2004} and \eqref{eq:b1}, \eqref{eq:b2}, and
\eqref{eq:b3} were first proved by Goulden, Nica and Oancea \cite[Equations
(4.6), (4.7), (4.9)]{GNO}.

\section{Annular noncrossing matchings} 
\label{sec:annul-noncr-match}

An \emph{$(n,m)$-annular noncrossing matching} is a complete matching on $[n+m]$
which can be drawn in an $(n,m)$-annulus without crossing.  By considering each
matching pair $(i,j)$ as a cycle of size $2$, one can identify an annular
noncrossing matching as an annular noncrossing permutation consisting of cycles
of size $2$ only.

\begin{thm}\label{thm:ANM}
  Suppose $n\equiv m\equiv c \mod 2$.  The number of $(n,m)$-annular noncrossing
  matchings exactly $c$ connected matching pairs is
\[
c \binom{n}{\frac{n-c}{2}} \binom{m}{\frac{m-c}{2}}.
\]
\end{thm}
\begin{proof}
  Such an annular noncrossing matching can be considered as an annular
  noncrossing permutation with $\alpha=(2^R)\in \Par(2R,R)$,
  $\beta=(2^S)\in \Par(2S,S)$, $\lambda=\mu=(1^c)\in \Par(c,c)$, and
  $n-2R=m-2S=c$.  Then we obtain the formula immediately from \eqref{eq:anc}.
\end{proof}

We can also obtain the total number of connected annular noncrossing matchings.

\begin{thm}
  For $n\equiv m\mod 2$, the number of connected $(n,m)$-annular noncrossing
  matchings is
\[
\frac{2\ceil{\frac n2}\ceil{\frac m2}}{n+m}
\binom{n}{\ceil{\frac n2}} \binom{m}{\ceil{\frac m2}}.
\]
\end{thm}
\begin{proof}
  We will prove the equivalent statement: for $k\in \{0,1\}$, the number
  of connected $(2n+k,2m+k)$-annular noncrossing matchings is
\[
\frac{(n+k)(m+k)}{n+m+k}
\binom{2n+k}{n+k} \binom{2m+k}{m+k}.
\]
By Theorem~\ref{thm:ANM} the number is equal to 
\[
\sum_{c\ge0} (2c+k) \binom{2n+k}{n-c}\binom{2m+k}{m-c}. 
\]
Then we are done by Lemma~\ref{lem:sum1} for $q=1$.
\end{proof}

\section*{Acknowledgement}
The author would like to thank Vic Reiner for providing his paper with Sommers
\cite{ReinerSommers}, suggesting the problem of finding a bicyclic sieving
phenomenon on annular noncrossing permutations, and for fruitful comments. He
also thanks Dennis Stanton for helpful discussion.


\end{document}